\documentclass[12pt]{article}
\usepackage{amsmath, latexsym, amsfonts, amssymb, amsthm, amscd}
\usepackage[left=2.5cm, right=2.5cm, top=2.5cm, bottom=2.5cm]{geometry}
\usepackage{upquote}
\usepackage{color}
\usepackage{setspace}
\usepackage[labelfont=bf]{caption}

\usepackage[utf8]{inputenc}
\usepackage[T1]{fontenc}
\usepackage[english]{babel}

\usepackage{tikz}	

	\tikzstyle{lien}=[->,>=stealth,rounded corners=2pt,thick]
	\tikzset{individu/.style={draw,#1},individu/.default={}}

\newtheorem{theorem}{Theorem}
\newtheorem{corollary}[theorem]{Corollary}
\newtheorem{proposition}[theorem]{Proposition}
\newtheorem{lemma}[theorem]{Lemma}
\newtheorem{remark}[theorem]{Remark}

\begin{document}
\title{The cut-tree of large trees with small heights}
\author{G. Berzunza\footnote{ {\sc Institut f\"ur Mathematik, Universit\"at Z\"urich, Winterthurerstrasse 190, CH-8057 Z\"urich, Switzerland;} e-mail: gabriel.berzunza@math.uzh.ch}}
\maketitle

\vspace{0.1in}

\begin{abstract} 
\noindent We destroy a finite tree of size $n$ by
cutting its edges one after the other and in uniform random order. Informally, the 
associated cut-tree describes the genealogy of the connected 
components created by this destruction process. We provide
a general criterion for the convergence of the rescaled cut-tree 
in the Gromov-Prohorov topology to an interval endowed with the Euclidean distance and 
a certain probability measure, when the underlying tree has branching points close to the root 
and height of order $o(\sqrt{n})$. In particular, we consider 
uniform random recursive trees, binary search trees, scale-free random trees
and a mixture of regular trees. This yields extensions of a result in Bertoin \cite{Be3} for the cut-tree of 
uniform random recursive trees and also allows us to generalize some 
results of Kuba and Panholzer \cite{Kuba} on the multiple isolation
of vertices. The approach relies in the close relationship between
the destruction process and Bernoulli bond percolation, 
which may be useful for studying the cut-tree of other classes of trees.
\bigskip 

\end{abstract}

\section{Introduction and main result}

\subsection{General introduction}

Consider a tree $T_{n}$ on a finite set of vertices, say $[n] := \{ 1, \dots, n\}$,
rooted at $1$. Imagine that we destroy it by cutting its edges one after the other, in a 
uniform random order. After $n-1$ steps, all edges have been destroyed and all the vertices are isolated. Meir and Moon 
\cite{MaM, MaM2} initiated the study of such procedure by considering the number of 
cuts required to isolate the root, when the edges are 
removed from the current component containing this distinguished vertex. More precisely, 
they estimated the first and second moments of this quantity for two important trees families, Cayley 
trees and random recursive trees. Concerning Cayley trees and other families of simply 
generated trees, a weak limit theorem for the number of cuts to isolate the root vertex 
was proven by Panholzer \cite{Panh} and, in greater generality by Janson \cite{Svante}
who also obtained the result for complete binary trees \cite{Svante2}. Holmgren 
\cite{Holmgren, Holmgren2} extended the approach of Janson to binary search trees and to
the family of split trees. For random recursive trees a limit law was obtained, first by
Drmota et al. \cite{Drmota} and reproved using a probabilistic approach by Iksanov and M\"ohle \cite{Iksanov}. 
\\

We observe that during the destruction process the cut of an edge induces the partition
of the subset (or block) that contains this edge into two sub-blocks of $[n]$. We then
encode the destruction of $T_{n}$ by a rooted binary tree, which we call the cut-tree 
and denote by $\text{Cut}(T_{n})$. The cut-tree has internal vertices given by the
non-singleton connected components which arise during the destruction, and leaves which
correspond to the singletons $ \{1\}, \dots, \{n\}$ (these can be identified as the
vertices of $T_{n}$). More precisely, the $\text{Cut}(T_{n})$ is rooted at the block $[n]$,
then we build it inductively: we draw an edge between a parent block $B$ and two
children blocks $B^{\prime}$ and $B^{\prime \prime}$ whenever an edge is removed from the 
subtree of $T_{n}$ with set of vertices $B$, producing two subtrees $B^{\prime}$ 
and $B^{\prime \prime}$. See Figure 1 for an illustration. \\

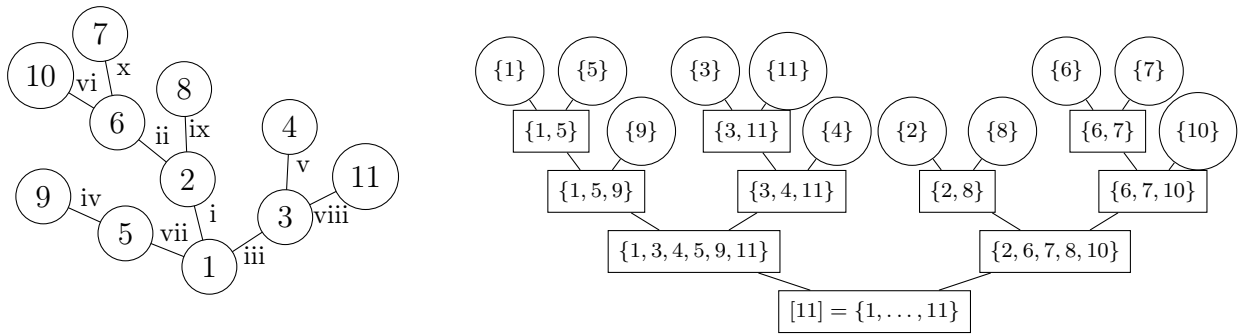
\begin{figure}[ht] \centering 
\begin{minipage}[b]{10\linewidth}
\begin{tikzpicture}[scale = 1.2]
\draw
	(0,0) -- (-0.24,0.97)			
	(-0.24,0.97) -- (-1.02,1.6)		
	(0,0) -- (0.84,0.55)			
	(-0.93,0.37) -- (-1.84,0.78)		
	(0.84,0.55) -- (0.89,1.55)		
	(-1.02,1.6) -- (-1.87,2.12)		
	(0,0) -- (-0.93,0.37)			
	(0.84,0.55) -- (1.73,1)		
	(-0.24,0.97) -- (-0.28,1.97)		
	(-1.02,1.6) -- (-1.21,2.58)		
;
\draw
	(0,0)			node [circle, fill=white, draw]	{1}
	(-0.24,0.97)	node [circle, fill=white, draw]	{2}
	(0.84,0.55)	node [circle, fill=white, draw]	{3}
	(0.89,1.55)	node [circle, fill=white, draw]	{4}
	(-0.93,0.37)	node [circle, fill=white, draw]	{5}
	(-1.02,1.6)		node [circle, fill=white, draw]	{6}
	(-1.21,2.58)	node [circle, fill=white, draw]	{7}
	(-0.28,1.97)	node [circle, fill=white, draw]	{8}
	(-1.84,0.78)	node [circle, fill=white, draw]	{9}
	(-1.87,2.12)	node [circle, fill=white, draw]	{10}
	(1.73,1)		node [circle, fill=white, draw]	{11}
;
\begin{footnotesize}
\draw
	(0.04,0.59)	node		{i}
	(-0.52,1.46)	node		{ii}
	(0.5,0.13)		node		{iii}
	(-1.31,0.77)	node		{iv}
	(1.04,1.11)	node		{v}
	(-1.36,2.05)	node		{vi}
	(-0.39,0.38)	node		{vii}
	(1.35,0.58)	node		{viii}
	(-0.11,1.54)	node		{ix}
	(-0.95,2.19)	node		{x}
;
\end{footnotesize}
\end{tikzpicture}
\end{minipage}
\ \ \hfill \begin{minipage}[b]{0.25\linewidth}
\vspace{-50cm}
\begin{tikzpicture}
[level 1/.style={sibling distance=4.8 cm},
level 2/.style={sibling distance=2.6 cm},
level 3/.style={sibling distance=1.2 cm},
level 4/.style={sibling distance=1.1 cm}]
%
\node[individu]{{\scriptsize$[11] = \{1,\dots,11\}$}} [grow'=up,level distance=.8cm]
	child{node[individu]{{\scriptsize$\{1,3,4,5,9,11\}$}}
		child{node[individu]{{\scriptsize$\{1,5,9\}$}}
			child{node[individu]{{\scriptsize $\{1,5\}$}}
				child{node[individu=circle]{{\scriptsize$\{1\}$}}}
				child{node[individu=circle]{{\scriptsize$\{5\}$}}}
			}
			child{node[individu=circle]{{\scriptsize$\{9\}$}}}
		}
		child{node[individu]{{\scriptsize$\{3,4,11\}$}}
			child{node[individu]{{\scriptsize$\{3,11\}$}}
				child{node[individu=circle]{{\scriptsize$\{3\}$}}}
				child{node[individu=circle]{{\scriptsize$\{11\}$}}}
			}
			child{node[individu=circle]{{\scriptsize$\{4\}$}}}
		}
	}
	child{node[individu]{{\scriptsize$\{2,6,7,8,10\}$}}
		child{node[individu]{{\scriptsize$\{2,8\}$}}
			child{node[individu=circle]{{\scriptsize$\{2\}$}}}
			child{node[individu=circle]{{\scriptsize$\{8\}$}}}
		}
		child{node[individu]{{\scriptsize$\{6,7,10\}$}}
			child{node[individu]{{\scriptsize$\{6,7\}$}}
				child{node[individu=circle]{{\scriptsize$\{6\}$}}}
				child{node[individu=circle]{{\scriptsize$\{7\}$}}}
			}
			child{node[individu=circle]{{\scriptsize$\{10\}$}}}
		}
	}
;
\end{tikzpicture}
\end{minipage}
\caption{{\small A tree of size eleven with the order of cuts on the left, and the corresponding
cut-tree on the right}}
\end{figure}

Roughly speaking, cut-trees describe the genealogy of connected 
components appearing in this edge-deletion process. They are especially useful 
in the study of the number of cuts needed to isolate any given subset of distinguished vertices, when the connected components which 
contain no distinguished points are discarded as soon as they appear. 
For instance, the number of cuts required to isolate $k$ distinct vertices 
$v_{1}, \dots, v_{k}$ coincides with the total length of the cut-tree reduced 
to its root and $k$ leaves $\{v_{1}\}, \dots, \{v_{k}\}$ minus $(k-1)$, where the length
is measured as usual by the graph distance on $\text{Cut}(T_{n})$. This 
motivated the study of the cut-tree for 
several families of trees. Bertoin \cite{Be} considered the cut-tree of Cayley trees, 
more generally, Bertoin and Miermont \cite{Be2} dealt with critical Galton-Watson trees with finite variance and conditioned to have size
$n$. More recently, Bertoin \cite{Be3} studied the uniform random recursive trees,
Dieuleveut \cite{Da} the Galton-Watson trees with offspring distribution belonging to the 
domain of attraction of a stable law of index $\alpha \in (1,2]$, and Broutin and 
Wang \cite{BoWa} the so-called $p$-trees. They described the asymptotic behavior (in distribution)
of the cut-trees when $n \rightarrow \infty$, for these classes of trees. We stress that in \cite{Be2, Da} the cut-tree slightly differs
from the one defined above, and in particular \cite{Da} considered a vertex removal 
procedure. \\ 

On the other hand, Baur \cite{bau2} has recently introduced another tree 
associated to the destruction process of uniform random recursive trees,
called \emph{tree of components}. Informally, one considers a dynamically version 
of the cutting procedure, where edges are equipped with i.i.d. exponential clocks and 
deleted at time given by the corresponding variable. Then, each removal of an edge gives
birth to a new tree component, whose sizes and birth times are encoding by a 
tree-indexed process. He used this tree of components to study cluster sizes created 
from performing Bernoulli bond percolation on uniform random recursive trees.
We do not study the tree of components in this work but, we think it would 
be of interest, and may be seen as a complement of the cut-tree. However, a common feature with our analysis is that, it is useful
to consider a continuous time version of the destruction process. \\

The main purpose of this work is study the behavior of
$\text{Cut}(T_{n})$ when the vertices of the underlying tree $T_{n}$ is star-shaped. 
Informally, we assume that the last common ancestor
of two randomly chosen vertices is close to the root, after proper rescaling, 
with high probability. We consider also that $T_{n}$ has a small height of order 
$o(\sqrt{n})$, in the sense that that the distance (the number of edges) between
its root $1$, and a typical vertex in $T_{n}$ is of this order $o(\sqrt{n})$. For instance, 
this is the case for uniform random recursive trees, binary search trees, scale-free random trees
and regular trees; see  for example Drmota \cite{Drmota2}, Barab\'asi \cite{Ba}, and 
Mahmound and Neininger \cite{Mah}. Informally, our main result provides a general criterion, depending on the nature of $T_{n}$, for
the convergence in distribution of the rescaled $\text{Cut}(T_{n})$ when $n \rightarrow \infty$.
\\

We next introduce the necessary notation and relevant 
background, which we will enable us to state our main result in Section \ref{main}.

\subsection{Measured metric spaces and the Gromov-Prokhorov topology}

We begin by introducing some basic facts about 
topological space of trees in which limits can be taken, and define the limit objects. A pointed metric measure space is a quadruple 
$(\mathcal{T}, d, \rho, \nu)$ where $(\mathcal{T}, d)$ is a separable and complete metric space, 
$\rho \in \mathcal{T}$ a distinguished element called the root of $\mathcal{T}$, and 
$\nu$ a Borel probability measure on $(\mathcal{T}, d)$.
This quadruple is called a real tree if in addition, $\mathcal{T}$ is a tree, in the 
sense that it is a geodesic space for which any two points are connected via a unique 
continuous injective path up to re-parametrization. This is a continuous analog of the 
graph-theoretic definition of a tree as a connected graph with no cycle. For sake of simplicity, we 
frequently write $\mathcal{T}$ to refer to a pointed metric measure space 
$(\mathcal{T}, d, \rho, \nu)$. We say that two measured rooted spaces $(\mathcal{T}, d, \rho, \nu)$ and
$(\mathcal{T}^{\prime}, d^{\prime}, \rho^{\prime}, \nu^{\prime})$ are isometry-equivalent if 
there exists a root-preserving, bijective isometry $\phi: \text{supp}(\mu) \cup \{\rho\} \rightarrow 
\mathcal{T}^{\prime}$ (here $\text{supp}$ is the topological support) such that 
the image of $\nu$ by $\phi$ is $\nu^{\prime}$. This defines an equivalence
relation between pointed metric measure spaces, and we note that representatives
$(\mathcal{T}, d, \rho, \nu)$ of a given isometry-equivalence class can always be assumed
to have $\text{supp}(\mu) \cup \{\rho\} = \mathcal{T}$. It is also convenient to agree that for $a>0$, 
$a \mathcal{T}$ denotes the same space $\mathcal{T}$ but with distance rescaled by the factor $a$, i.e.
$(\mathcal{T}, ad, \rho, \nu)$. \\

It is well-known that the set $\mathbb{M}$ of isometry-equivalence classes of pointed
metric spaces is a Polish space when endowed with the so-called Gromov-Prokhorov topology.
This topology was introduced by Greven, Pfaffelhuber and Winter in \cite{Greven}
under the name of {\sl Gromov-weak} topology. We also refer to Gromov's book \cite{Gromov}, 
the article of Haas and Miermont \cite{Haas} and references therein 
for background. We can then view the $\text{Cut}(T_{n})$ for $n \geq 1$ as a sequence random variables with values in $\mathbb{M}$
(i.e. a sequence of real random tree). For convenience, we adopt a slightly different point of view for $\text{Cut}(T_{n})$ than the usual for finite trees, focusing on leaves rather than internal nodes. More precisely,
we set $[n]^{0}= \{ 0,1, \dots, n\}$ where $0$ correspond to the root $[n]$ of 
$\text{Cut}(T_{n})$ and $1, \dots, n$ to the leaves (i.e. $i$ is identified with the singleton $\{ i \}$). We consider the random pointed metric 
measure space $([n]^{0}, \delta_{n}, 0, \mu_{n})$ where $\delta_{n}$ is the random graph distance
on $[n]^{0}$ induced by the cut-tree, $0$ is the distinguished element, and $\mu_{n}$
is the uniform probability measure on $[n]$ extended by $\mu_{n}(0) =0$. That is, $\mu_{n}$
is the uniform probability measure on the set of leaves of $\text{Cut}(T_{n})$. We 
point out that the combinatorial structure of the cut-tree can be recovered from
$([n]^{0}, \delta_{n}, 0, \mu_{n})$, so by a slight abuse of notation, sometimes 
we refer to $\text{Cut}(T_{n})$ as the latter pointed metric measure space. \\

Finally, we recall a convenient characterization of the Gromov-Prokhorov topology that relies on the 
convergence of distances between random points. A  sequence 
$(\mathcal{T}_{n}, d_{n}, \rho_{n}, \nu_{n})$ of pointed
measure metric spaces converges in the Gromov-Prokhorov sense to an element of 
$\mathbb{M}$, say $(\mathcal{T}_{\infty}, d_{\infty}, \rho_{\infty}, \nu_{\infty})$,
if and only if the following holds: for $n \in \{1,2, \dots \} \cup \{\infty \}$, set $\xi_{n}(0) = \rho_{n}$ and let $\xi_{n}(1), \xi_{n}(2), \dots$ be a sequence of i.i.d. random variables with law $\nu_{n}$, then 
\begin{eqnarray*} 
(d_{n}(\xi_{n}(i), \xi_{n}(j)): i,j \geq 0) \Rightarrow (d_{\infty}(\xi_{\infty}(i), \xi_{\infty}(j)): i, j \geq 0) 
\end{eqnarray*}

\noindent where $\Rightarrow$ means convergence in the sense of finite-dimensional distribution, $\xi_{\infty}(0) = \rho_{\infty} $
and $\xi_{\infty}(1), \xi_{\infty}(2), \dots$ is a sequence of i.i.d. random variables with law $\nu_{\infty}$; see for example Corollary 8
of \cite{Lo}. One can interpret $(d_{\infty}(\xi_{\infty}(i), \xi_{\infty}(j)): i, j \geq 0)$ as the matrix
of mutual distances between the points of an i.i.d. sample of $(\mathcal{T}_{\infty}, d_{\infty}, \rho_{\infty}, \nu_{\infty})$.
Moreover, it is important to point out that by the Gromov's reconstruction theorem
in \cite{Gromov}, the distribution of the above matrix of distances characterizes $(\mathcal{T}_{\infty}, d_{\infty}, \rho_{\infty}, \nu_{\infty})$
as an element of $\mathbb{M}$.

\subsection{Main result} \label{main}

We first introduce notation and hypotheses which will have an important
role for the rest of the work. Recall that $T_{n}$ is a tree with set of vertices
$[n] = \{ 1, \dots, n\}$, rooted at $1$. We denote by $u$ and $v$ two independent uniformly 
distributed random vertices on $[n]$. Let
$d_{n}$ be the graph distance in $T_{n}$, and  $\ell: \mathbb{N} \rightarrow \mathbb{R}_{+}$ be some function
such that $\lim_{n \rightarrow \infty} \ell(n) = \infty$. We introduce the following hypothesis
\begin{equation} \label{cond3}
  \frac{1}{\ell(n)} (d_{n}(1,u), d_{n}(u,v)) \Rightarrow (\zeta_{1}, \zeta_{1} + \zeta_{2}). \tag{$H$}
\end{equation}

\noindent where $\zeta_{1}$ and $\zeta_{2}$ are i.i.d. variables in $\mathbb{R}_{+}$ with
no atom at $0$. This happens with $\zeta_{i}$ a positive constant for some important families of random 
trees, such as uniform recursive trees, regular trees, scale-free random trees and binary search trees (and more generally
$b$-ary recursive trees). In Section \ref{exa}, we consider a different class of examples where the variable $\zeta_{i}$ 
is not a constant, which results of the mixture of similar trees satisfying the hypothesis (\ref{cond3}). 

\begin{remark} \label{rem1}
We observe that 
\begin{eqnarray*}
 d_{n}(u,v) = d_{n}(1,u) + d_{n}(1,v) - 2d_{n}(1, u \wedge v),
\end{eqnarray*}

\noindent where $u \wedge v$ is the last common ancestor of $u$ and $v$ in $T_{n}$. 
Then, the condition (\ref{cond3}) readily implies that $\lim_{n \rightarrow \infty} \ell(n)^{-1} d_{n}(1, u \wedge v) = 0$
in probability. Moreover, if for each fixed $k \in \mathbb{N}$, we denote by $L_{k,n}$ the length
of the tree $T_{n}$ reduced to $k$ vertices chosen uniformly at random with replacement 
and its root $1$, i.e. the minimal number of edges of $T_{n}$ which are needed to connect 
$1$ and such vertices, we see that (\ref{cond3}) is equivalent to 
\begin{equation*} 
 \frac{1}{\ell(n)} (L_{1,n}, L_{2,n}) \Rightarrow (\zeta_{1}, \zeta_{1} + \zeta_{2}).
\end{equation*}
\end{remark}

We then write
\begin{eqnarray*}
 \lambda(t) = \mathbb{E} [e^{-t \zeta_{1}}], \hspace*{5mm} \text{for} \hspace*{2mm} t \geq 0,
\end{eqnarray*}

\noindent for the Laplace transform of the random variable $\zeta_{1}$.
We henceforth denote
\begin{eqnarray*}
a = \mathbb{E}[1/\zeta_{1}],
\end{eqnarray*}

\noindent which can be infinite. We define the bijective mapping $\Lambda: [0, \infty) \rightarrow [0, a)$ by
\begin{eqnarray*}
 \Lambda(t) = \int_{0}^{t} \lambda(s) {\rm d}s, \hspace*{5mm} \text{for} \hspace*{3mm} t \geq 0,
\end{eqnarray*}

\noindent where $\Lambda(\infty) = \lim_{t \rightarrow \infty} \Lambda(t) = a$, and 
write $\Lambda^{-1}$ for its inverse mapping. Observe that (\ref{cond3}) entails that
\begin{equation*}
  \frac{1}{\ell(n)}  d_{n}(u,v) \Rightarrow  \zeta_{1} + \zeta_{2},
\end{equation*}

\noindent then we consider the next technical condition 
\begin{equation} \label{cond2}
 \lim_{n \rightarrow \infty} \mathbb{E} \left[ \frac{\ell(n)}{d_{n}(u, v)} \mathbf{1}_{\{ u \neq v \}} \right] = \mathbb{E}\left[\frac{1}{\zeta_{1} + \zeta_{2}}  \right]< \infty.  \tag{$H^{\prime}$}
\end{equation}


 \begin{theorem} \label{teo1}
 Suppose that (\ref{cond3}) and (\ref{cond2}) hold with $\ell$ such that
 $\ell(n) = o(\sqrt{n})$. Furthermore, assume that $a < \infty$. Then as $n \rightarrow \infty$, we have the following convergence 
 in distribution in the sense of the pointed Gromov-Prokhorov topology:
 \begin{eqnarray*}
  \frac{\ell(n)}{n} \text{{\rm Cut}}(T_{n}) \Rightarrow I_{\mu}.
 \end{eqnarray*}
 \noindent where $I_{\mu}$ is the pointed measure metric space given by the interval $[0,a]$, pointed at $0$, equipped
with the Euclidean distance, and the probability measure $\mu$ given by 
\begin{eqnarray} \label{ec2}
 \int_{0}^{a} f(x) \mu({\rm d} x) = - \int_{0}^{a} f(x) \, {\rm d} \lambda \circ \Lambda^{-1}(x)
\end{eqnarray}

\noindent where $f$ is a generic positive measurable function. The result still valid when $a=\infty$, and then one considers the interval $[0, \infty)$, pointed at $0$, equipped with the same distance and measure. 
\end{theorem}

We stress that Theorem \ref{teo1} does not apply for the family of 
critical Galton-Watson trees conditioned to have size $n$ considered for
Bertoin and Miermont \cite{Be2} and Dieuleveut \cite{Da} since they do not 
satisfy the condition (\ref{cond3}), and the height of a typical vertex 
is not of the order $o(\sqrt{n})$. For instance, the case when $T_{n}$ is a Cayley tree
(conditioned Galton-Watson tree with Poisson offspring distribution), for which it is 
know that $\ell(n) = \sqrt{n}$ and the variable $L_{i,n}$ in Remark \ref{rem1}, for $i=1,2$,
is a chi-variable with $2k$ degrees of freedom; see for example Aldous \cite{A}. We believe that the 
threshold $\sqrt{n}$ appearing in this work is critical, and that for 
trees with larger heights (of order $\Omega(\sqrt{n})$ following Knut's definition) the limit of their
rescaled cut-tree is a random tree, and not a deterministic one. For instance, in the case when $T_{n}^{(c)}$ is a Cayley tree of 
size $n$, it has been shown in \cite{Be} that $n^{-1/2} \text{Cut}(T_{n}^{(c)})$ converges in distribution to a 
Brownian Continuum Random tree, in the sense of Gromov-Hausdorff-Prokhorov.
This uses crucially a general limit theorem due to Haas and Miermont \cite{Haas} for
so-called Markov branching trees. This has been extended in \cite{Be2} 
to a large family of critical Galton-Watson trees with finite variance, and by  
Dieuleveut \cite{Da} when the offspring distribution belongs to the domain
of attraction of a stable law of index $\alpha \in (1, 2]$, both in the sense
of Gromov-Prokhorov. We point out that in \cite{Da} the limit is a 
stable random tree of index $\alpha$. \\

On the other hand, it has been shown in \cite{Be3} for a 
uniform random recursive tree $T_{n}^{(r)}$
of size $n$ that upon rescaling the graph distance of $\text{Cut}(T_{n}^{(r)})$ by a
factor $n^{-1}\ln n$, the latter converges in probability
in the sense of pointed Gromov-Hausdorff-Prokhorov distance to the unit interval
$[0,1]$ equipped with the Euclidean distance and the Lebesgue measure, and pointed at $0$.
The basic idea in \cite{Be3} for establishing the result for uniform random recursive trees 
relies crucially on a coupling due to Iksanov and M\"ohle \cite{Iksanov} that connects the destruction
process in this family of trees with a remarkable random walk. However, this coupling
is not fulfilled in general for the trees we are interested in, and we thus have to use a fairly different 
route. \\

Loosely speaking, our approach relies on the introduction of
a continuous version of the cutting down procedure, where edges are equipped with i.i.d.
exponential random variables and removed at a time given by the corresponding variable. Following 
Bertoin \cite{Be5} we represent the destruction process up to 
a certain finite time as a Bernoulli bond-percolation, allowing us to relate the tree 
components with percolation clusters. We then develop the ideas in \cite{Be5} used to 
analyze cluster sizes in supercritical percolation, and study the asymptotic behavior
of the process that counts the number of edges which are remove from the root as 
time passed, which is closely related with the distance induced by the cut-tree. \\

The plan of the rest of this paper is as follows. Section \ref{sec2} is devoted 
to the continuous-time version of the destruction procedure on a general
random tree, which will play a crucial role in our analysis of the cut-tree. We then
establish our main result Theorem \ref{teo1} in Section \ref{sec3}. In Section
\ref{exa}, we provide some examples of trees that fulfill the hypotheses (\ref{cond3}) and 
(\ref{cond2}). Then in Section \ref{sec4} we present some applications
on the isolation of multiple vertices, which extend the results
of Kuba and Panholzer \cite{Kuba}, and Baur and Bertoin \cite{Bau} for uniform
random recursive trees. Section \ref{sec5} is devoted to the proof of a technical result
about the shape of scale-free random trees, which may be of independent interest. 



\section{Cutting down in continuous time} \label{sec2}

The purpose of this section is to study the destruction dynamics on a general 
sequence of random trees $T_{n}$. We consider a continuous time version of the 
destruction process in which edges are removed independently one of the others
at a given rate. We establish the link with Bernoulli bond-percolation and deduce
some properties related to the destruction process, which will 
be relevant for the proof of Theorem \ref{teo1}. \\

Recall that for each fixed $k \in \mathbb{N}$, we denote by $L_{k,n}$ the length
of the tree $T_{n}$ reduced to $k$ vertices chosen uniformly at random with replacement 
and its root $1$. Recall also the Remark \ref{rem1} and then consider the following weaker version of the hypothesis (\ref{cond3}),
\begin{equation} \label{cond}
 \frac{1}{\ell(n)} L_{k,n} \Rightarrow \zeta_{1} + \dots + \zeta_{k}, \tag{$H_{k}$}
\end{equation}

\noindent where $\zeta_{1}, \dots$ is a sequence of i.i.d. variables in $\mathbb{R}_{+}$ with
no atom at $0$, and the convergence in (\ref{cond}) is in the sense of one-dimensional distribution,
i.e. for each fixed $k$. We stress that the hypothesis (\ref{cond3}) implies (\ref{cond}) for $k =1,2$. \\

We then present the continuous time version of the destruction process. We attach to each edge 
$e$ of $T_{n}$ an independent exponential random variable $\mathbf{e}(e)$ of parameter 
$1/\ell(n)$, and we delete it at time $\mathbf{e}(e)$. After the $(n-1)$th edge has been 
deleted, the tree has been destructed, and the process ends. Rigorously, let $e_{1}, \dots, e_{n-1}$
denote the edges of $T_{n}$ listed in the increasing order of their attached exponential
random variables, i.e. such that $\mathbf{e}(e_{1}) < \dots < \mathbf{e}(e_{n-1})$. Then at time 
$\mathbf{e}(e_{1})$, the first edge $e_{1}$ is removed from $T_{n}$, and $T_{n}$ splits into two 
subtrees, say $\tau_{n}^{1}$ and $\tau_{n}^{*}$, where $\tau_{n}^{1}$ contains the root $1$.
Next, if $e_{2}$ connects two vertices in $\tau_{n}^{*}$ then at time $\mathbf{e}(e_{2})$, 
$\tau_{n}^{*}$ splits in two tree components. Otherwise, $\tau_{n}^{1}$ splits in two subtrees
after removing the edge $e_{2}$. We iterate in an obvious way until all the vertices of $T_{n}
$ have been isolated. \\

Define $p_{n}(t) = \exp (-t/ \ell(n))$ for $t \geq 0$, and observe that the probability 
that a given edge has not yet been removed at time $t$ in the continuous time destruction process
is $p_{n}(t)$. Thus, the configuration observed at time $t$ is precisely that resulting 
from a Bernoulli bond percolation on $T_{n}$ with parameter $p_{n}(t)$. Further, Bertoin \cite{Be4} proved that
when the hypothesis (\ref{cond}) is fulfilled  for $k=1,2$, the percolation parameter $p_{n}(t)$ 
corresponds to the supercritical regime, in the sense that with 
high probability, there exists a giant cluster, that is of size (number of vertices) comparable to that of the
entire tree. Thus focusing on the evolution of the tree component which contains the root
$1$, we write $X_{n}(t)$ for its size at time $t \geq 0$; plainly $X_{n}(t) \leq n$. 
We shall establish the following limit theorem which is an improvement of Corollary 1 (i) in \cite{Be4}.

\begin{proposition} \label{cor2}
 Suppose that (\ref{cond}) holds for $k =1,2$. Then, we have that
 \begin{eqnarray} \label{ec6}
  \lim_{n \rightarrow \infty} \sup_{ s \geq 0} |n^{-1}X_{n}(s) - \lambda(s)| = 0 \hspace*{6mm} \text{in probability}.
 \end{eqnarray}
\end{proposition}

\begin{proof}
 It follows from Corollary 1(i) in \cite{Be4} that for $t \geq 0$
\begin{eqnarray*}
 \lim_{n \rightarrow \infty} n^{-1} X_{n}(t) = \lambda(t) \hspace*{6mm} \text{in probability},
\end{eqnarray*}
where $\lambda(t) = \mathbb{E}(e^{-t \zeta_{1}})$ for $t \geq 0$, when ever (\ref{cond}) holds for $k =1,2$.
Then by the diagonal procedure, we may extract from an arbitrary 
increasing sequence of integers a subsequence, say $(n_{l})_{l \in \mathbb{N}}$, such 
that with probability one,
\begin{eqnarray*} 
   \lim_{l \rightarrow \infty} n^{-1}_{l}X_{n_{l}}(s) = \lambda(s) \hspace*{6mm} \text{for all rational} \hspace*{2mm} s \geq 0.
\end{eqnarray*}

\noindent As $s \rightarrow X_{n}(s)$ decreases, and $s \rightarrow \lambda(s)$ is continuous, the above convergence holds uniformly on 
$[0,t]$ for an arbitrary fixed $t > 0$, i.e.
\begin{eqnarray} \label{ecc1}
  \lim_{l \rightarrow \infty} \sup_{0 \leq s \leq t} |n^{-1}_{l}X_{n_{l}}(s) - \lambda(s)| = 0 \hspace*{6mm} \text{a.s.}. 
 \end{eqnarray}

On the other hand, we observe that $\lim_{s \rightarrow \infty} \lambda(s) = 0$.
Then for $\varepsilon > 0$, we can find $t_{\varepsilon} > 0$ and $N(\varepsilon) >0$ such that
\begin{eqnarray*}
  \sup_{ s > t_{\varepsilon}} |n^{-1}_{l}X_{n_{l}}(s) - \lambda(s)| < \varepsilon \hspace*{6mm} \text{for} \hspace*{2mm} n_{l} > N(\varepsilon), \hspace*{4mm} \text{a.s.},
 \end{eqnarray*}

\noindent and therefore, our claim follows by combining (\ref{ecc1}) and the above observation.
\end{proof}

It is interesting to recall that the reciprocal of Proposition \ref{cor2} holds. 
More precisely, Corollary 1 (ii) in \cite{Be4} shows that (\ref{cond}), for $k =1,2$, form
a necessarily and sufficient condition for (\ref{ec6}). \\

In order to make the connexion with the discrete destruction process introduced at the beginning of
this work, which is the one we are interested in, we now turn our attention to the number 
$R_{n}(t)$ of edges of the current root component which have been removed up to time $t$ in the 
procedure described above. We observe that every jump of the process $R_{n} = (R_{n}(t): t \geq 0)$ corresponds
to removing an edge from the root component according to the discrete destruction process. 
We interpret the latter as a continuous time version  of a random algorithm introduced
by Meir and Moon \cite{MaM, MaM2} for the isolation of the root. Recall also that 
\begin{eqnarray*}
 \Lambda(t) = \int_{0}^{t}\lambda(s) {\rm ds}, \hspace*{5mm} \text{for} \hspace*{2mm} t \geq 0.
\end{eqnarray*}

\begin{lemma} \label{lema1}
Suppose that (\ref{cond}) holds for $k =1,2$, with $\ell$ such that $\ell(n) = o(\sqrt{n})$. Then, we have for
 every fixed $t >0$
 \begin{eqnarray*}
  \lim_{n \rightarrow \infty} \sup_{0 \leq s \leq t} \left| \frac{\ell(n)}{n} R_{n}(s) - \Lambda(s) \right| = 0 \hspace*{6mm} \text{in probability}.
 \end{eqnarray*}
\end{lemma}

\begin{proof}
 We denote by $X_{n} = (X_{n}(t): t \geq 0)$ the process
of the size of the root cluster. The dynamics of the continuous time destruction process show that the 
counting process $R_{n}$ grows at rate $\ell(n)^{-1}(X_{n}-1)$, which means 
rigorously that the predictable compensator of $R_{n}(t)$ is absolutely continuous with 
respect to the Lebesgue measure with density $\ell(n)^{-1}(X_{n}(t)-1)$. In other words,
\begin{eqnarray*}
 M_{n}(t) = R_{n}(t) - \int_{0}^{t} \ell(n)^{-1}(X_{n}(s)-1) {\rm d}s
\end{eqnarray*}

\noindent is a martingale; note also that its jumps $|M_{n}(t)-M_{n}(t-)|$ have size
at most $1$. Since there are at most $n-1$ jumps up to time $t$, the bracket of 
$M_{n}$ can be bounded by $[M_{n}]_{t} \leq n-1$. By Burkholder–Davis–Gundy inequality, 
we have that
\begin{eqnarray*}
 \mathbb{E}[|M_{n}(t)|^{2}] \leq n-1,
\end{eqnarray*}

\noindent and in particular, since we assumed that $\ell(n) = o(\sqrt{n})$,
\begin{eqnarray} \label{ec5}
 \lim_{n \rightarrow \infty} \mathbb{E} \left[ \left| \frac{\ell(n)}{n} M_{n}(t) \right|^{2} \right] = 0.
\end{eqnarray}

On the other hand, since (\ref{cond}) holds for $k = 1,2$, Proposition \ref{cor2} 
and dominated convergence entail
\begin{eqnarray*}
 \lim_{n \rightarrow \infty}  \frac{\ell(n)}{n}\int_{0}^{t} \ell(n)^{-1}(X_{n}(s)-1) {\rm d}s =  \int_{0}^{t} \lambda(s) {\rm d}s 
 \hspace*{5mm} \text{in probability}.
\end{eqnarray*}

\noindent Hence from (\ref{ec5}) we have
that
\begin{eqnarray*}
 \lim_{n \rightarrow \infty}  \frac{\ell(n)}{n}R_{n}(t) =   \Lambda(t)  \hspace*{5mm} \text{in probability},
\end{eqnarray*}

\noindent and since $t \rightarrow R_{n}(t)$ increases, by the diagonal procedure as 
in the proof of Proposition \ref{cor2}, our claim follows. 
\end{proof}

We continue our analysis of the destruction process, and prepare the ground for 
the main result of this section, which is the estimation of the number 
of steps in the algorithm for the isolating the root which are needed to
disconnect (and not necessarily isolate) a vertex chosen uniformly at random 
from the root component. We start by studying the analogous quantity in continuous time. For each fixed $n \in \mathbb{N}$, we denote by $u_{1}, u_{2}, \dots$ a sequence of
i.i.d. vertices in $[n] = \{1, \dots, n \}$ with the uniform distribution. Next, for every $i \in \mathbb{N}$,
we write $\Gamma_{i}^{(n)}$ the first instant when the vertex $u_{i}$ is disconnected from the root component. 
We shall establish the following limit theorem in law.

\begin{proposition} \label{pro3}
 Suppose that (\ref{cond}) holds for $k = 1,2$. Then 
 as $n \rightarrow \infty$, the random vector
 \begin{eqnarray*} 
  (\Gamma_{i}^{(n)}: i \geq 1) \Rightarrow (\gamma_{i} : i \geq 1)
 \end{eqnarray*}
in the sense of finite-dimensional distribution, where $\gamma_{1}, \gamma_{2}, \dots$ 
are i.i.d. random variables in $\mathbb{R}_{+}$ with distribution given by 
$\mathbb{P}(\gamma_{1} > t) = \lambda(t)$ for $t \geq 0$. 
\end{proposition}

\begin{proof}
 We observe that for every $j \in \mathbb{N}$ and 
$t_{1}, \dots, t_{j} \geq 0$, there is the identity
\begin{eqnarray*}
 \mathbb{P}(\Gamma_{1}^{(n)} > t_{1}, \dots, \Gamma_{j}^{(n)} > t_{j}) = \mathbb{P}(u_{1} \in T_{n}^{(1)}(t_{1}), \dots, u_{j} \in T_{n}^{(1)}(t_{j})), 
\end{eqnarray*}

\noindent where $T_{n}^{(1)}(t)$ denotes the subtree at time $t$ which contains the root $1$. Recall that $u_{1}, \dots, u_{j}$ are i.i.d. uniformly 
distributed vertices, which are independent of the destruction process. 
On the other hand, for $t \geq 0$ the variable $n^{-1}X_{n}(t)$ is the 
proportion of vertices in the root component at time $t$, and represents the 
conditional probability that a vertex of $T_{n}$ chosen uniformly at random 
belongs to the root component at time $t$. We thus have
\begin{eqnarray*}
 \mathbb{P}(\Gamma_{1}^{(n)} > t_{1}, \dots, \Gamma_{j}^{(n)} > t_{j}) = \mathbb{E} \left[  n^{-j} \prod_{i=1}^{j} X_{n}(t_{i}) \right]. 
\end{eqnarray*}

\noindent Since (\ref{cond}) holds for $k = 1,2$,
we conclude from Proposition \ref{cor2} that
\begin{eqnarray*}
 \lim_{n \rightarrow \infty} \mathbb{P}(\Gamma_{1}^{(n)} > t_{1}, \dots, \Gamma_{j}^{(n)} > t_{j}) = \prod_{i=1}^{j}  \lambda(t_{i}), 
\end{eqnarray*}

\noindent which establishes our claim. 
\end{proof}

We are now in position to state the main result of this section. We provide a non-trivial 
limit in distribution for the number $Y_{i}^{(n)}$ of cuts (in the algorithm for 
isolating the root) which are needed to disconnect a vertex chosen uniformly at random, 
say $u_{i}$, from the root component.

\begin{corollary} \label{cor3}
 Suppose that (\ref{cond}) holds for $k = 1,2$, with $\ell$
 such that $\ell(n) = o(\sqrt{n})$. 
 Then as $n \rightarrow \infty$, we have that
 \begin{eqnarray*} 
  \left( \frac{\ell(n)}{n} Y_{i}^{(n)}: i \geq 1 \right) \Rightarrow (Y_{i} : i \geq 1)
 \end{eqnarray*}
\noindent in the sense of finite-dimensional distribution, where $Y_{1}, Y_{2}, \dots$ 
are i.i.d. random variables on $[0,a)$ where $a = \Lambda(\infty)$, and with distribution 
given by 
\begin{eqnarray} \label{ec1}
 \mathbb{E}[f(Y_{1})]= - \int_{0}^{a} f(x) {\rm d} \lambda \circ \Lambda^{-1}(x), 
\end{eqnarray}

\noindent where $f$ is a generic positive measurable function.
\end{corollary}

\begin{proof} 
Recall that $R_{n}(t)$ denotes the number of edges of the root component which have been
removed up to time $t$ in the continuous procedure described above. We recall also 
that $\Gamma_{i}^{(n)}$ denotes the first instant when the vertex 
$u_{i}$, chosen uniformly at random, has been disconnected from the root component. 
Hence we have the following identity,
\begin{eqnarray*}
 Y_{i}^{(n)} = R_{n}(\Gamma_{i}^{(n)})  \hspace*{5mm} \text{for} \hspace*{2mm} i \in \mathbb{N}.
\end{eqnarray*}

\noindent It follows from Lemma \ref{lema1} and Proposition \ref{pro3} that
\begin{eqnarray*}
 \lim_{n \rightarrow \infty} \left( \frac{\ell(n)}{n} R_{n}(\Gamma_{i}^{(n)}) - \Lambda(\Gamma_{i}^{(n)}) \right) = 0 \hspace*{6mm} \text{in probability}, 
\end{eqnarray*}

\noindent and therefore, as $n \rightarrow \infty$, we have that
 \begin{eqnarray*} 
  \left( \frac{\ell(n)}{n} Y_{i}^{(n)}: i \geq 1 \right) \Rightarrow \left( \Lambda(\gamma_{i}) : i \geq 1 \right)
 \end{eqnarray*}

 \noindent in the sense of finite-dimensional distribution, where 
 $\gamma_{1}, \gamma_{2}, \dots$ are i.i.d. random variables in $\mathbb{R}_{+}$ with
 distribution given by $\mathbb{P}(\gamma_{1} > t ) = \lambda(t)$. Finally, we only need
 to verify that the law of $\Lambda(\gamma_{1})$ is given
 by (\ref{ec1}). We observe that by dominated convergence $\lambda$ is differentiable,
 and we denote by $\lambda^{\prime}$ its derivative. Then for $f$ a generic positive measurable
 function that
 \begin{eqnarray*}
  \mathbb{E} \left[ f \left( \Lambda(\gamma_{1}) \right)\right] & = & -\int_{0}^{\infty} f \left( \Lambda( x) \right) \lambda^{\prime}(x)  {\rm d} x. 
 \end{eqnarray*}
 
 On the other hand, we observe that $\Lambda$ is an increasing continuous and
 differentiable function whose derivative is never $0$. Hence
  \begin{eqnarray*}
  \mathbb{E} \left[ f \left( \Lambda(\gamma_{1}) \right)\right] & = & -\int_{0}^{\Lambda(\infty)} f \left( x \right) \frac{\lambda^{\prime} \circ \Lambda^{-1}(x)}{\lambda \circ \Lambda^{-1}(x)} {\rm d} x \\
  & = & -\int_{0}^{\Lambda(\infty)} f \left( x \right)  {\rm d} \lambda \circ \Lambda^{-1}(x), 
 \end{eqnarray*}
 
 \noindent which completes the proof.
 \end{proof}

Corollary \ref{cor3} will have a crucial role in the proof of 
Theorem \ref{teo1}. This result will enable us to get a precise estimate of 
distances in the cut-tree.  \\

Finally, let $N^{(u)}(n)$ be the number of remaining cuts that is needed to isolate a 
vertex chosen uniformly at random, say $u$, once it has been disconnected from the root component.
The next proposition establishes a criterion which ensures that $N^{(u)}(n)$ is small compared to 
$n /\ell(n)$ with high probability. This technical ingredient will be useful later on in the proof of 
Theorem \ref{teo1}.

\begin{proposition} \label{cor5}
 Assume that (\ref{cond3}) and (\ref{cond2}) hold with $\ell$ such that $\ell(n) = o(\sqrt{n})$. Then we have
 \begin{eqnarray*}
  \lim_{n \rightarrow \infty} \frac{\ell(n)}{n} N^{(u)}(n) = 0 \hspace*{5mm} \text{in probability.}
 \end{eqnarray*}
\end{proposition}

\begin{proof}
 We write $R_{n}^{(u)}(t)$ for the number of edges that have been removed up to time $t$
from the tree component containing the vertex $u$, and $\Gamma_{n}$ the first instant when the vertex $u$ has been 
disconnected from the root cluster; in particular,
\begin{eqnarray*}
 \lim_{t \rightarrow \infty} R_{n}^{(u)}(\Gamma_{n} + t) - R_{n}^{(u)}(\Gamma_{n}) = N^{(u)}(n).
\end{eqnarray*}

\noindent Let $X_{n}^{(u)}(t)$ be the size of the subtree containing the vertex $u$
at time $t$. Since each edge is removed with rate $\ell(n)^{-1}$, independently
of the other edges, the process
\begin{eqnarray*}
 M_{n}^{(u)}(t) =  R_{n}^{(u)}(\Gamma_{n} + t) - R_{n}^{(u)}(\Gamma_{n}) - \int_{0}^{t} \ell(n)^{-1} (X_{n}^{(u)}(\Gamma_{n} + s) -1) {\rm d}s, \hspace*{5mm} t \geq 0,
\end{eqnarray*}

\noindent is a purely discontinuous martingale with terminal value
\begin{eqnarray*}
 \lim_{t \rightarrow \infty} M_{n}^{(u)}(t) =  N^{(u)}(n) - \int_{0}^{\infty} \ell(n)^{-1} (X_{n}^{(u)}(\Gamma_{n} + s) -1) {\rm d}s.
\end{eqnarray*}

\noindent Further, its bracket can be bounded by $[M_{n}^{(u)}]_{t} \leq n-1$. Then since we
assume that $\ell(n) = o(\sqrt{n})$,
\begin{eqnarray*}
 \lim_{n \rightarrow \infty} \mathbb{E} \left[ \left| \frac{\ell(n)}{n} N^{(u)}(n) - \int_{0}^{\infty} n^{-1} (X_{n}^{(u)}(\Gamma_{n} + s) -1) {\rm d}s \right|^{2} \right] = 0.
\end{eqnarray*}

\noindent Therefore, it only remains to prove that 
\begin{eqnarray} \label{ec10}
 \lim_{n \rightarrow \infty} \mathbb{E} \left[ \int_{0}^{\infty} n^{-1} (X_{n}^{(u)}(\Gamma_{n} +s) -1) {\rm d}s \right]=0.
\end{eqnarray}

\noindent Let $T^{(u)}_{n}(s)$ denote the subtree at time $s$ which contains the vertex $u$. We observe that 
\begin{eqnarray*}
 \mathbb{E} \left[ \int_{0}^{\infty} n^{-1} (X_{n}^{(u)}(\Gamma_{n}+s) -1) {\rm d}s \right] & = & \mathbb{E} \left[ \int_{0}^{\infty} n^{-1} (X_{n}^{(u)}(s) -1) \mathbf{1}_{\{\Gamma_{n} \leq s \}} {\rm d}s \right] \\
 & = &  \mathbb{E} \left[ \int_{0}^{\infty} n^{-1} (X_{n}^{(u)}(s) -1) \mathbf{1}_{ \left\{1 \notin T^{(u)}_{n}(s) \right\}} {\rm d}s \right].
\end{eqnarray*}

\noindent We note that a vertex $v$ chosen uniformly at random in $[n]$ and independent of $u$ 
belong to the same cluster at time $t$ if and only if no edge on the path form $u$ and $v$ has been
removed at time $t$. Recall that the probability that a given edge has not yet been removed at time
$t$ is $\exp(-t/\ell(n))$ in the continuous time destruction process. Recall that $d_{n}$ denotes the graph distance in $T_{n}$, 
and $u \wedge v$ the last common ancestor of $u$ and $v$. Then, we have that
\begin{eqnarray*}
\mathbb{E} \left[ n^{-1}(X_{n}^{(u)}(t)-1) \mathbf{1}_{\left \{1 \notin T^{(u)}(t) \right\}} \right] & = & n^{-1} \mathbb{E} \left[ \sum_{i \in [n] \setminus u} \mathbf{1}_{ \left\{i \in T^{(u)}_{n}(t), \,  1 \notin T^{(u)}_{n}(t) \right\}} \right] \\
& = &  \mathbb{E} \left[ \left( e^{-\frac{d_{n}(u,v)}{\ell(n)} t} - e^{-\frac{L_{2,n}}{\ell(n)} t} \right)\mathbf{1}_{\{v \neq u\}}\right],
\end{eqnarray*}

\noindent where $L_{2,n}$ is 
the length of the tree $T_{n}$ reduced to the vertex $u,v$ and its root. Then,
\begin{eqnarray} \label{ec15}
 \mathbb{E} \left[ \int_{0}^{\infty} n^{-1} (X_{n}^{(u)}(\Gamma_{n}+s) -1) {\rm d}s \right] & = & \mathbb{E} \left[ \left(\frac{\ell(n)}{d_{n}(u,v)} - \frac{\ell(n)}{L_{2,n}} \right) \mathbf{1}_{\{v \neq u\}}\right].
 \end{eqnarray}

\noindent On the other hand, since
\begin{eqnarray*}
  \frac{\ell(n)}{L_{2,n}} \mathbf{1}_{\{ v \neq u \}} \leq \frac{\ell(n)}{d_{n}(u,v)} \mathbf{1}_{\{ v \neq u \}}, 
\end{eqnarray*}

\noindent it is not difficult to see from Remark \ref{rem1} that the assumption (\ref{cond2}) implies that
\begin{equation*} 
 \lim_{n \rightarrow \infty} \mathbb{E} \left[ \frac{\ell(n)}{L_{2,n}} \mathbf{1}_{\{ v \neq u \}} \right] = \mathbb{E}\left[\frac{1}{\zeta_{1} + \zeta_{2}}  \right]< \infty.
 \end{equation*}

\noindent Therefore, we get (\ref{ec10}) by letting $n \rightarrow \infty$ in (\ref{ec15}).
\end{proof}

\section{Proof of Theorem \ref{teo1}} \label{sec3}

In this section, we prove our main result, Theorem \ref{teo1}. We stress that during 
the proof we consider that the tree $T_{n}$ is a deterministic tree. This will clearly 
imply the result for random trees. In this direction,
we recall that we view the $\text{Cut}(T_{n})$ as the pointed metric measure space 
$([n]^{0}, \delta_{n}, 0, \mu_{n})$, where $0$ corresponds to the root and
$1, \dots, n$ to the leaves, $\delta_{n}$ the
graph distance induced by the cut-tree, and $\mu_{n}$ the uniform probability measure on 
$[n]$ with $\mu_{n}(0) = 0$. We assume that $a = \Lambda(\infty) < \infty$. 
We then recall that $I_{\mu}$ denotes the pointed measure metric space
given by the interval $[0, a]$, pointed at $0$, equipped with the 
Euclidean distance, and the probability measure $\mu$ given in (\ref{ec2}), i.e.
\begin{eqnarray*} 
 \int_{0}^{a} f(x) \mu({\rm d} x) = - \int_{0}^{a} f(x) \, {\rm d} \lambda \circ \Lambda^{-1}(x),
\end{eqnarray*}

\noindent where $f$ is a generic positive measurable function. We stress 
that in the case $a = \infty$ the proof follows along the same lines as that of 
$a < \infty$. Then, $I_{\mu}$ denotes the pointed measure metric space
given by the interval $[0, \infty)$, pointed at $0$, equipped with the Euclidean 
distance and the measure $\mu$. \\

We recall that to establish weak convergence in the sense induced by the Gromov-Prokhorov
topology, we shall prove the convergence in distribution of the rescaled distances of 
$\text{Cut}(T_{n})$. Specifically, for every $n \in \mathbb{N}$, set $\xi_{n}(0)=0$ and 
consider a sequence $(\xi_{n}(i))_{i \geq 1}$ of i.i.d. random variables with law 
$\mu_{n}$. We will prove that
\begin{eqnarray*}
\left( \frac{\ell(n)}{n} \delta_{n}(\xi_{n}(i), \xi_{n}(j)): i,j \geq 0 \right) \Rightarrow (\delta(\xi(i), \xi(j)): i,j \geq 0) 
\end{eqnarray*}

\noindent in the sense of finite-dimensional distribution, where $\xi(0) = 0$ and
$(\xi(i))_{i \geq 1}$ is a sequence of i.i.d. random variables on $\mathbb{R}_{+}$ with law $\mu$.
Furthermore, $\delta(\xi(i), \xi(j)) = |\xi(i) - \xi(j)|$ since $\delta$ is the
Euclidean distance, and in particular,  $\delta(0, \xi(i)) = \xi(i)$.\\

The key idea of the proof relies in the relationship between the distance 
in $\text{Cut}(T_{n})$, and the number of cuts needed to disconnect certain 
number of vertices in $T_{n}$. Indeed, the height of the leaf $\{i\}$ in $\text{Cut}(T_{n})$
is precisely the number of cuts needed to isolate the vertex $i$ in $T_{n}$. Therefore, it 
will be convenient to think in $(\xi_{n}(i))_{i \geq 1}$ as a sequence of i.i.d. vertices
in $[n]$, with the uniform distribution. 

\begin{proof}[ Proof of Theorem \ref{teo1}] We observe that for $i \geq 1$,
\begin{eqnarray*}
 \delta_{n}(\xi_{n}(0) , \xi_{n}(i)) = \delta_{n}(0 , \xi_{n}(i))
\end{eqnarray*}

\noindent is precisely the number of cuts which are needed to isolate the vertex
$\xi_{n}(i)$. For each $n \in \mathbb{N}$,
we denote by $\delta_{n}^{(1)}(0, \xi_{n}(i))$ the number of cuts which are needed to disconnect 
the vertex $\xi_{n}(i)$ from the root component, and by $\eta(\xi_{n}(i))$ the remaining
number of cuts which are needed to isolate the vertex $\xi_{n}(i)$ after it has been disconnected. Clearly, 
we have
\begin{eqnarray*}
 \delta_{n}(0 , \xi_{n}(i)) - \delta_{n}^{(1)}(0 , \xi_{n}(i)) = \eta(\xi_{n}(i)).
\end{eqnarray*}

\noindent Since the condition (\ref{cond2}) holds, Proposition \ref{cor5} implies that $\lim_{n \rightarrow \infty} n^{-1}\ell(n) \eta(\xi_{n}(i)) = 0$
in probability for $i \geq 1$. Therefore, the assumption (\ref{cond3}) entails according to 
Corollary \ref{cor3} that
\begin{eqnarray*}
 \left( \frac{\ell(n)}{n} \delta_{n}(0 , \xi_{n}(i)): i \geq 0 \right) \Rightarrow \left( \xi(i): i \geq 0 \right)
\end{eqnarray*}

\noindent in the sense of finite-dimensional distribution. Essentially, we follow the same argument to show that the preceding also holds
jointly with
 \begin{eqnarray} \label{ec8}
 \left( \frac{\ell(n)}{n} \delta_{n}(\xi_{n}(i) , \xi_{n}(j)): i,j \geq 1 \right) \Rightarrow \left( \delta(\xi(i) , \xi(j)): i,j \geq 1 \right)
\end{eqnarray}

\noindent which is precisely our statement. \\

In this direction, for $i,j \geq 1$, we denote by $\delta_{n}^{(2)}(\xi_{n}(i) , \xi_{n}(j))$ the number
of cuts which are needed to isolate the vertices $\xi_{n}(i)$ and $\xi_{n}(j)$.
We also write $\delta_{n}^{(3)}(\xi_{n}(i) , \xi_{n}(j))$ for the number of cuts 
(in the algorithm for isolating the root) until for the first time, the vertices $\xi_{n}(i)$ and $\xi_{n}(j)$ 
are disconnected. Hence from the description of the cut-tree, 
it should be plain that
\begin{eqnarray} \label{ec9}
 \delta_{n}(\xi_{n}(i) , \xi_{n}(j)) = (\delta_{n}^{(2)}(\xi_{n}(i) , \xi_{n}(j))+1) - (\delta_{n}^{(3)}(\xi_{n}(i) , \xi_{n}(j))-1).
\end{eqnarray}

\noindent Next we observe that
\begin{eqnarray*}
 \delta_{n}^{(3)}(\xi_{n}(i) , \xi_{n}(j)) - \min(\delta_{n}^{(1)}(0 , \xi_{n}(i)), \delta_{n}^{(1)}(0 , \xi_{n}(j))) \leq \eta(\xi_{n}(i)) + \eta(\xi_{n}(j)),
\end{eqnarray*}

\noindent and 
\begin{eqnarray*}
 \delta_{n}^{(2)}(\xi_{n}(i) , \xi_{n}(j)) - \max(\delta_{n}^{(1)}(0 , \xi_{n}(i)), \delta_{n}^{(1)}(0 , \xi_{n}(j))) \leq \eta(\xi_{n}(i)) + \eta(\xi_{n}(j)).
\end{eqnarray*}

\noindent Since the assumption (\ref{cond3}) and (\ref{cond2}) hold, it follows from Proposition \ref{cor5} that
\begin{eqnarray*}
 \lim_{n \rightarrow \infty} \frac{\ell(n)}{n} \left( \eta(\xi_{n}(i)) + \eta(\xi_{n}(j)) \right) = 0 \hspace*{6mm} \text{in probability}.
\end{eqnarray*}

\noindent Moreover, Corollary \ref{cor3} implies that 
 \begin{eqnarray*}
 \left( \frac{\ell(n)}{n} \delta_{n}^{(3)}(\xi_{n}(i) , \xi_{n}(j)): i,j \geq 1 \right) \Rightarrow \left( \min(\xi(i) , \xi(j)): i,j \geq 1 \right),
\end{eqnarray*}

\noindent and
 \begin{eqnarray*} 
 \left( \frac{\ell(n)}{n} \delta_{n}^{(2)}(\xi_{n}(i) , \xi_{n}(j)): i,j \geq 1 \right) \Rightarrow \left( \max(\xi(i) , \xi(j)): i,j \geq 1 \right)
\end{eqnarray*}

\noindent hold jointly. Therefore, since $\delta$ is the Euclidean distance, 
the convergence in (\ref{ec8}) follows from the identity (\ref{ec9}).
\end{proof}

\section{Examples} \label{exa}

In this section, we present some examples of trees that fulfilled the conditions of Theorem
\ref{teo1}. But first, we observe that when the hypotheses of the 
latter are satisfied with $\zeta_{1} \equiv 1$,
the probability measure $\mu$ given in (\ref{ec2}) corresponds to the Lebesgue measure on the unit interval
$[0,1]$. The above follows from the fact that $\lambda(t) = e^{-t}$ for all $t \geq 0$. Then we have the following
interesting consequence of Theorem \ref{teo1}.

\begin{corollary} \label{cor4} 
  Suppose that (\ref{cond3}) and (\ref{cond2}) hold, with $\zeta_{1} \equiv 1$ 
  and $\ell$ such that $\ell(n) = o(\sqrt{n})$.
 Then as $n \rightarrow \infty$, we have the following convergence 
 in the sense of the pointed Gromov-Prokhorov topology:
 \begin{eqnarray*}
  \frac{\ell(n)}{n} \text{{\rm Cut}}(T_{n}) \Rightarrow I_{1}.
 \end{eqnarray*}
 \noindent where $I_{1}$ is the pointed measure metric space given by
 the unit interval $[0,1]$, pointed at $0$, equipped with the Euclidean 
 distance and the Lebesgue measure.
\end{corollary}

A natural example is the class of random trees with logarithmic heights, i.e. which fulfill 
hypothesis (\ref{cond3}) with $\ell(n) = c \ln n$ for some $c >0$, such as binary search trees, 
regular trees, uniform random recursive trees, and more generally scale-free random trees. 
We are now going to prove that (\ref{cond2}) is also satisfied for the previous families 
of trees and therefore their rescaled cut-tree converges in the sense of Gromov-Prokhorov topology
to $I_{1}$.

\paragraph*{1. Binary search trees.} A popular family of random trees used in computer 
science for sorting and searching data is the binary search tree. More precisely, a 
binary search tree is a binary tree in which each vertex is associated to a key, 
where the keys are drawn randomly from an ordered set, we say $\{1, \dots, n\}$, until
the set is exhausted. The first key is associated to the root. The next key is 
placed at the left child of the root if it is smaller than the root's key and
placed to the right if it is larger. Then one proceeds progressively, inserting key by key.
When all the keys are placed one gets a binary tree with $n$ vertices. For further details, see
e.g. \cite{Drmota2}. Theorem S1 in Devroye \cite{D22} shows that
the hypothesis (\ref{cond3}) holds with $\ell(n) = 2 \ln n$. Hence in order
to be in the framework of Corollary \ref{cor4} all that we need 
is to check that this family of trees fulfills the hypothesis (\ref{cond2}), namely
\begin{eqnarray*} 
 \lim_{n \rightarrow \infty} \mathbb{E} \left[ \frac{2 \ln n}{d_{n}(u,v)} \mathbf{1}_{\{u \neq v \}} \right]= \frac{1}{2},
\end{eqnarray*}

\noindent where $u$ and $v$ are two vertices chosen uniformly at random with replacement 
from the binary search tree of size $n$. 
In this direction, we pick $0 < \varepsilon < (2 \ln 2)^{-1}$ and consider the function
$\phi_{\varepsilon}$ given by $\phi_{\varepsilon} = 0$ on $[0 , \varepsilon]$,
$\phi_{\varepsilon} =1$ on $[2 \varepsilon, \infty)$, and $\phi_{\varepsilon}$ linear on 
$[\varepsilon, 2\varepsilon]$. We observe that 
\begin{eqnarray*}
 \lim_{n \rightarrow \infty} \mathbb{E} \left[ \frac{ 2 \ln n}{d_{n}(u,v)} \phi_{\varepsilon} \left( \frac{d_{n}(u,v)}{ 2 \ln n} \right) \mathbf{1}_{\{u \neq v \}} \right] =  \frac{1}{2}\phi_{\varepsilon}\left(\frac{1}{2} \right). 
\end{eqnarray*}

\noindent Further, we note that $\phi_{\varepsilon}(1/2) \rightarrow 1$ 
as $\varepsilon \rightarrow 0$. Then, it is enough to prove that 
\begin{eqnarray} \label{ec7}
\lim_{\varepsilon \rightarrow 0} \limsup_{n \rightarrow \infty}\mathbb{E} \left[ \left( \frac{2 \ln n}{d_{n}(u,v)} - \frac{ 2 \ln n}{d_{n}(u,v)} \phi_{\varepsilon} \left( \frac{d_{n}(u,v)}{ 2 \ln n} \right) \right) \mathbf{1}_{\{u \neq v \}} \right] = 0,
\end{eqnarray}

\noindent in order to show (\ref{cond2}). We write $X^{i}(n,k)$ for the number of
vertices at distance $k \geq 1$ from the vertex $i$ in a binary search tree of size $n$.
Then 
\begin{eqnarray*}
 \mathbb{E} \left[ \left( \frac{2 \ln n}{d_{n}(u,v)} - \frac{ 2 \ln n}{d_{n}(u,v)} \phi_{\varepsilon} \left( \frac{d_{n}(u,v)}{ 2 \ln n} \right) \right) \mathbf{1}_{\{u \neq v \}} \right] & \leq & \mathbb{E} \left[ \frac{ 2 \ln n}{d_{n}(u,v)}  \mathbf{1}_{\{d_{n}(u,v) \leq 2 \varepsilon \ln n, \, u \neq v  \}}  \right] \\
 & \leq & \frac{2 \ln n}{n^{2}} \sum_{i=1}^{n} \sum_{k=1}^{\lfloor 2 \varepsilon \ln n \rfloor} \frac{1}{k} \mathbb{E}[X^{i}(n,k)].
\end{eqnarray*}

\noindent Since each vertex in a binary search tree has at most
two descendants, we observe that $\mathbb{E}[X^{i}(n,k)] \leq 3 \cdot 2^{k-1}$. Then
\begin{eqnarray*}
 \mathbb{E} \left[ \left( \frac{2 \ln n}{d_{n}(u,v)} - \frac{ 2 \ln n}{d_{n}(u,v)} \phi_{\varepsilon} \left( \frac{d_{n}(u,v)}{ 2 \ln n} \right) \right) \mathbf{1}_{\{u \neq v \}} \right] & \leq & \frac{3 \ln n}{n} \sum_{k=1}^{\lfloor 2 \varepsilon \ln n \rfloor}  \frac{2^{k}}{k} \\
 & \leq & \frac{6 \ln n}{n} 2^{2 \varepsilon \ln n},
\end{eqnarray*}

\noindent and therefore we get (\ref{ec7}) by letting $n \rightarrow \infty$ and $\varepsilon \rightarrow 0$. \\

\noindent More generally, one can consider a generalization of the binary search trees, 
namely the $b$-ary recursive trees and check that these fulfill the conditions of 
Corollary \ref{cor4} with $\ell(n) = \frac{b}{b-1} \ln n$; we refer to 
Devroye \cite{luc}.

\paragraph*{2. Scale free random trees.} The scale-free random trees form a family of random trees that grow following a 
preferential attachment algorithm, and are used commonly to model complex real-word 
networks; see Barab\'asi and Albert \cite{Ba}. Specifically, fix a parameter $\alpha \in (-1, \infty)$, and 
start for $n=1$ from the  tree $T_{1}^{(\alpha)}$ on $\{1,2\}$ which has a single edge 
connecting $1$ and $2$. Suppose that  $T_{n}^{(\alpha)}$ has been constructed for some 
$n \geq 2$, and for every $i \in \{1, \dots, n+1\}$, denote by $g_{n}(i)$ the degree 
of the vertex $i$ in $T_{n}^{(\alpha)}$. Then conditionally given $T_{n}^{(\alpha)}$, the tree
$T_{n+1}^{(\alpha)}$ is built by adding an edge between the new vertex $n+2$ and a vertex 
$v_{n}$ in $T_{n}^{(\alpha)}$ chosen at random according to the law
\begin{eqnarray*}
 \mathbb{P} \left( v_{n} = i| T_{n}^{(\alpha)} \right) = \frac{g_{n}(i)+ \alpha}{2n+\alpha(n+1)}, \hspace*{5mm} i \in \{1, \dots, n+1\}.
\end{eqnarray*}

\noindent We observe that when one lets $\alpha \rightarrow \infty$ the algorithm yields
an uniform recursive tree. It is not difficult to check that the condition (\ref{cond3}) in 
Corollary \ref{cor4} is fulfilled with $\ell(n) = \frac{1+\alpha}{2+\alpha} \ln n$; see for instance 
\cite{Be6}. Then, it only remains to check the hypothesis (\ref{cond2}). We only prove the 
latter when $\alpha=0$, the general case follows similarly but with longer computations.
We then follow the same route as the case of the binary search trees. Pick $\varepsilon > 0$ and consider the same
function $\phi_{\varepsilon}$ that we defined previously. Therefore, it is enough to show that
\begin{eqnarray*}
\lim_{\varepsilon \rightarrow 0} \limsup_{n \rightarrow \infty}\mathbb{E} \left[ \left( \frac{  \ln n}{2 d_{n}(u,v)} - \frac{  \ln n}{2 d_{n}(u,v)} \phi_{\varepsilon} \left( \frac{2 d_{n}(u,v)}{  \ln n} \right) \right) \mathbf{1}_{\{u \neq v \}} \right] = 0,
\end{eqnarray*}

\noindent where $u$ and $v$ are two independent uniformly
distributed random vertices on $T_{n}^{(0)}$ . We observe that
\begin{eqnarray} \label{ec12}
 \mathbb{E} \left[ \left( \frac{ \ln n}{2 d_{n}(u,v)} - \frac{  \ln n}{2 d_{n}(u,v)} \phi_{\varepsilon} \left( \frac{2 d_{n}(u,v)}{  \ln n} \right) \right) \mathbf{1}_{\{u \neq v \}} \right]  \leq  \mathbb{E} \left[ \frac{ \ln n}{2 d_{n}(u,v)}  \mathbf{1}_{\{d_{n}(u,v) \leq \frac{1}{2} \varepsilon \ln n, \, u \neq v \}}  \right].
\end{eqnarray}

\noindent We write $Z^{i}(n,k)$ for the number of vertices
at distance $k \geq 1$ from the vertex $i$. Then, 
\begin{eqnarray*}
 \mathbb{E} \left[ \frac{  \ln n}{2 d_{n}(u,v)} \mathbf{1}_{\{u \neq v \}} \mathbf{1}_{\{d_{n}(u,v) \leq  \frac{1}{2} \varepsilon \ln n \}}  \right] & \leq & \frac{ \ln n}{ n^{2}} \sum_{i=1}^{n+1} \sum_{k=1}^{\lfloor  \frac{1}{2} \varepsilon \ln n \rfloor}  \frac{1}{k} \mathbb{E}[Z^{i}(n,k)]\\
 & \leq & \frac{ \ln n}{n^{2}} z^{ -\frac{1}{2} \varepsilon \ln n} \sum_{i=1}^{n+1} \mathbb{E}[G_{n}^{i}(z)], 
 \end{eqnarray*}
 
\noindent for $z  \in (0,1)$, where $G_{n}^{i}(z) = \sum_{k=0}^{\infty}  z^{k} Z^{i}(n,k+1)$. We claim the following.
\begin{lemma} \label{lema3}
 There exists $z_{0} \in (0,1)$ such that we have that
 \begin{eqnarray*}
  \mathbb{E}[G_{n}^{i}(z_{0})] \leq  e^{\frac{1+z_{0}}{2}} n^{\frac{1+z_{0}}{2}}, \hspace*{5mm} \text{for} \hspace*{2mm} i \geq 1 \hspace*{2mm} \text{and} \hspace*{2mm} n \geq 1.
 \end{eqnarray*}
 \end{lemma}
 
\noindent The proof of the above lemma relies in the recursive structure of the scale-free random tree and 
for now it is convenient to postpone its proof to Section \ref{sec5}. We then consider $z_{0}$ such that the result of Lemma \ref{lema3} holds and $0 < \varepsilon < (z_{0}-1) (\ln z_{0})^{-1}$. Then
\begin{eqnarray*}
 \mathbb{E} \left[ \frac{  \ln n}{2 d_{n}(u,v)} \mathbf{1}_{\{u \neq v \}} \mathbf{1}_{\{d_{n}(u,v) \leq  \frac{1}{2} \varepsilon \ln n \}}  \right] & \leq & e^{\frac{1+z_{0}}{2}} z_{0}^{ -\frac{1}{2} \varepsilon \ln n} n^{-\frac{1-z_{0}}{2}} \ln n
 \end{eqnarray*}

\noindent and therefore, the right-hand side in (\ref{ec12}) tends to $0$ as $n \rightarrow \infty$. \\

\noindent Similarly, one can easily check that the uniform random recursive trees fulfill
the hypotheses of Corollary \ref{cor4} with $\ell(n) = \ln n$; see Chapter 6 in \cite{Drmota2}.

\paragraph*{3. Merging of regular trees.} Our next example provides a method to
build trees that fulfill the conditions of Theorem \ref{teo1} and where
the random variable $\zeta_{1}$ in hypothesis (\ref{cond3}) is not a constant. Basically, the
procedure consists on gluing trees which satisfy the assumptions of Corollary
\ref{cor4}. In this example, we consider a mixture of regular trees but one may consider other 
families of trees as well.
For a fixed integer $r \geq 1$, let $(d_{i})_{i=1}^{r}$ denote a positive sequence of 
integers. Next, for $i = 1, \dots, r$, let $h_{i}(m): \mathbb{R}_{+} \rightarrow \mathbb{R}_{+}$
be a function with $\lim_{m \rightarrow \infty} h_{i}(m) = \infty$. Moreover, we 
assume that
\begin{eqnarray*}
 d_{1}^{h_{1}(m)} \sim d_{2}^{h_{2}(m)} \sim \dots \sim d_{r}^{h_{r}(m)}, 
\end{eqnarray*}

\noindent when $m \rightarrow \infty$. Then, let $T_{n_{i}}^{(d_{i})}$ be a 
complete $d_{i}$-regular tree with height $\lfloor h_{i}(m) \rfloor$. 
Since there are $d_{i}^{j}$ vertices at distance $j=0, 1, \dots, \lfloor h_{i}(m) \rfloor$
from the root, its size is given by
\begin{eqnarray*}
 n_{i} = n_{i}(m) = d_{i} (d_{i}^{\lfloor h_{i}(m) \rfloor}-1)/(d_{i}-1).
\end{eqnarray*}

\noindent In particular, one can check that the assumptions in Theorem \ref{teo1} are
fulfilled with $\ell(n_{i}) = \ln n_{i}$. We now imagine that we merge all
the $r$ regular trees into one common root which leads us to a new tree $T_{n}^{(d)}$
of size $n = \sum_{i=1}^{r} n_{i} + 1 -r$. Then, we observe that the probability 
that a vertex of $T_{n}^{(d)}$ chosen uniformly at random belongs to the tree $T_{n_{i}}^{(d_{i})}$
converges when $m \rightarrow \infty$ to $1/r$. Then, one readily checks that this new tree 
satisfies the hypothesis (\ref{cond3}) with $\ell(n) = \ln n$ and $\zeta_{i}$ a random variable uniformly distributed in the set
$\{1/ \ln d_{1}, \dots, 1 / \ln d_{r} \}$. Furthermore, 
since the number of descendants of each vertex is bounded, it is not difficult
to see that also fulfills the condition (\ref{cond2}). Therefore, Theorem
\ref{teo1} implies that $ n^{-1} \ln n \, \text{{\rm Cut}}(T_{n}^{(d)})$
converges in distribution in the sense of pointed Gromov-Prokhorov to
the element $I_{\mu^{(d)}}$ of $\mathbb{M}$, which corresponds to the interval $[0,a)$, 
pointed at $0$, equipped with the Euclidean distance, and the probability measure 
$\mu^{(d)}$ given by (\ref{ec2}) with $\lambda(t) = \frac{1}{r} \sum_{i=1}^{r} e^{- \frac{t}{\ln d_{i}}}$ for $t \geq 0$.

\section{Applications} \label{sec4}

We now present a consequence of Theorem \ref{teo1} which generalizes a result of 
Kuba and Panholzer \cite{Kuba}, and its recent multi-dimensional extension shown by 
Baur and Bertoin \cite{Bau} on the isolation of multiple vertices in uniform random recursive trees.
Let $u_{1}, u_{2}, \dots$ denote a sequence of i.i.d. uniform random
 variables in $[n] = \{1, \dots, n\}$. We write $Z_{n,j}$ for the number of
 cuts which are needed to isolate $u_{1}, \dots, u_{j}$ in $T_{n}$. We have the following
 convergence which extends Corollary 4 in \cite{Bau}.
 
 \begin{corollary} \label{cor6}
  Suppose that (\ref{cond3}) and (\ref{cond2}) hold with $\ell$ such that
 $\ell(n) = o(\sqrt{n})$. Then as $n \rightarrow \infty$, we have that
  \begin{eqnarray*}
   \left( \frac{\ell(n)}{n} Z_{n,j} : j \geq 1 \right) \Rightarrow (\max(U_{1}, U_{2}, \dots, U_{j}): j \geq 1)
  \end{eqnarray*}
\noindent in the sense of finite-dimensional distributions, where $U_{1}, U_{2}, \dots$  is
a sequence of i.i.d. random variables with law $\mu$ given in (\ref{ec2}). 
 \end{corollary}
 
\begin{proof} For a fixed integer $j \geq 1$, $u_{1}, \dots, u_{j}$ are
$j$ independent uniform vertices of $T_{n}$, or equivalently, the singletons
$\{u_{1}\}, \dots, \{u_{j}\}$ form a sequence of $j$ i.i.d. leaves of ${\rm Cut}(T_{n})$
distributed according the uniform law. Denote by $\mathcal{R}_{n,j}$ 
the subtree of ${\rm Cut}(T_{n})$ spanned by its root and $j$ i.i.d. leaves chosen 
according to the uniform distribution on $[n]$. Similarly, write $\mathcal{R}_{j}$ 
for the subtree of $I_{\mu}$ spanned by $0$ and $j$ i.i.d. random variables with law $\mu$, say $U_{1}, \dots, U_{j}$. We adopt the framework of Aldous 
\cite{A}, and see both reduced trees as combinatorial trees structure with edge lengths.
Therefore, Theorem \ref{teo1} entails that $ n^{-1} \ell(n) \mathcal{R}_{n,j}$ converges
weakly in the sense of Gromov-Prokhorov to $\mathcal{R}_{j}$ as $n \rightarrow \infty$. In 
particular, we have the convergence of the lengths of those reduced trees,
\begin{eqnarray*}
  \left( \frac{\ell(n)}{n} |\mathcal{R}_{n,1}|, 
 \dots, \frac{\ell(n)}{n} |\mathcal{R}_{n,j}| \right) \Rightarrow \left(|\mathcal{R}_{1}|, 
 \dots, |\mathcal{R}_{j}| \right).
\end{eqnarray*}
\noindent It is sufficient to observe that $|\mathcal{R}_{j}| = \max (U_{1}, \dots, U_{j} )$.
\end{proof}

In particular, when the hypotheses (\ref{cond3}) and (\ref{cond2}) hold with $\zeta_{1} \equiv 1$,
we observe from Corollary \ref{cor4} that the variables 
$U_{1}, U_{2}, \dots$ have the uniform distribution on $[0,1]$, 
and moreover, $\frac{\ell(n)}{n} Z_{n,j}$ converges
in distribution to a beta$(j,1)$ random variable. \\

As another application, for $j \geq 2$ we consider the algorithm for isolating the vertices $u_{1}, \dots, u_{j}$
with a slight modification, we discard the emerging tree components which 
contain at most one of these $j$ vertices. We stop the algorithm
when the $j$ vertices are totally disconnected from each other, i.e. lie in $j$
different tree components. We write $W_{n,2}$ for the number of steps of this algorithm until for the first time 
$u_{1}, \dots, u_{j}$ do not longer belong to the same tree component, moreover
$W_{n,3}$ for the number of steps until the first time, the $j$ vertices are spread out
over three distinct tree components, and so on, up to $W_{n, j}$, the number of steps
until the $j$ vertices are totally disconnected. We have the following
consequence of Corollary \ref{cor2}, which extends Corollary 4 in \cite{Bau}.

\begin{corollary}
 Suppose that (\ref{cond3}) and (\ref{cond2}) hold with $\ell$ such that
 $\ell(n) = o(\sqrt{n})$. Then as $n \rightarrow \infty$, we have that
  \begin{eqnarray*}
   \left( \frac{\ell(n)}{n} W_{n,2}, \dots, \frac{\ell(n)}{n} W_{n,j} \right) \Rightarrow \left( U_{(1,j)}, \dots, U_{(j-1,j)} \right),
  \end{eqnarray*}
\noindent where $U_{(1, j)} \leq U_{(2,j)} \leq \cdots \leq U_{(j-1,j)}$ 
denote the first $j-1$ order statistics of an i.i.d. sequence
$U_{1}, \dots, U_{j}$ of random variables with law $\mu$ given in (\ref{ec2}). 
\end{corollary}

\begin{proof} Recall the notation of Corollary \ref{cor2}, and
write $Y_{i}^{(n)}$ for the number of cuts which are needed to disconnect the
vertex $u_{i}$ from the root component. We then 
observe that if we write $Y^{(n)}_{1, j} \leq Y^{(n)}_{2, j} \leq \cdots \leq Y^{(n)}_{j-1, j}$
for the first order statistics of the sequence of random variables 
$Y_{1}^{(n)}, \dots, Y_{j}^{(n)}$, it follows from Proposition \ref{cor5} that 
\begin{eqnarray*}
 \lim_{n \rightarrow \infty}  \frac{\ell(n)}{n} ( W_{n, i} - Y^{(n)}_{i-1, j}) = 0 \hspace*{5mm} \text{ in probability.}
\end{eqnarray*}

\noindent Therefore, our claim follows
immediately from Corollary \ref{cor2}. 
\end{proof}

As before, when  (\ref{cond3}) and (\ref{cond2}) hold with $\zeta_{1} \equiv 1$,
the variables $U_{1}, U_{2}, \dots$ have the uniform distribution on $[0,1]$, 
and then, $\frac{\ell(n)}{n} W_{n,j}$ converges
in distribution to a beta$(1,j)$ random variable, and 
$\frac{\ell(n)}{n} W_{n,j}$ converges in distribution to a beta$(j-1,2)$ law.

\section{Proof of Lemma \ref{lema3}} \label{sec5}

The purpose of this final section is to establish Lemma \ref{lema3}. The proof 
relies on the recursive structure of the scale-free random trees, and our
guiding line is similar to that in \cite{Ka} and \cite{Ka2}. We recall that we only consider
the case when the parameter $\alpha$ of the scale-free random tree is zero, but that 
the general case can be treated similarly. \\

Recall that the construction of the scale-free tree starts at $n=1$ from the tree $T_{1}^{(0)}$ on $\{1,2\}$ which has a single edge 
connecting $1$ and $2$. Suppose that  $T_{n}^{(0)}$ has been constructed for some 
$n \geq 2$, then conditionally given $T_{n}^{(0)}$, the tree $T_{n+1}^{(0)}$ is built 
by adding an edge between the new vertex $n+2$ and a vertex 
$v_{n}$ in $T_{n}^{(0)}$ chosen at random according to the law
\begin{eqnarray*}
\mathbb{P} \left( v_{n} = i| T_{n}^{(0)} \right) = \frac{g_{n}(i)}{2n}, \hspace*{5mm} i \in \{1, \dots, n+1\}.
\end{eqnarray*}

\noindent where $g_{n}(i)$ denotes the degree of the vertex $i$ in $T_{n}^{(0)}$. Let $Z^{i}(n, k)$ denote the number of vertices at distance $k \geq 0$ from the vertex $i$
after the $n$-th step. We are interested in the expectation of the generating 
function 
\begin{eqnarray*}
 G_{n}^{i}(z) = \sum_{k = 0}^{\infty} Z^{i}(n, k+1) z^{k}, \hspace*{5mm} n\geq 1,
\end{eqnarray*}

\noindent for $z \in (0,1)$. In particular, $G_{n}^{1}(\cdot)$ is the so-called height profile function; see
Katona \cite{Ka, Ka2} for several results related to this function. 
To compute $\mathbb{E}[ G_{n}^{i}(z)]$ we use the evolution process of the construction
of $T_{n}^{(0)}$ and conditional expectation. Let $\mathcal{F}_{n}$ denote 
the $\sigma$-field generated by the first $n$ steps in the procedure. The number of vertices at distance
$k$ from $i$ increases by one or does not change. Then for $n \geq i-1$,
\begin{eqnarray*}
 \mathbb{E}[Z^{i}(n+1, 1) | \mathcal{F}_{n}] = (Z^{i}(n, 1) + 1) \frac{Z^{i}(n, 1)}{2n} + Z^{i}(n, 1) \left(1 -\frac{Z^{i}(n, 1)}{2n} \right) = \frac{2n +1 }{2n}Z^{i}(n, 1),
\end{eqnarray*}

\noindent and for $k > 1$ we have 
\begin{align*} \label{ec13}
  & \mathbb{E}[Z^{i}(n+1, k) | \mathcal{F}_{n}]  \\
  & ~~~~= (Z^{i}(n, k) + 1) \frac{Z^{i}(n, k) + Z^{i}(n, k-1) }{2n}   + Z^{i}(n, k) \left(1 -\frac{Z^{i}(n, k) + Z^{i}(n, k-1)}{2n} \right) \\
  &  ~~~~ =  \frac{2n+1}{2n} Z^{i}(n, k) + \frac{1}{2n}Z^{i}(n, k-1), 
\end{align*}

\noindent where $Z^{1}(0, k) = 0$ and $Z^{i}(i-2, k) = 0$ for $2 \leq i \leq n+1 $. Taking the 
expectation this leads to the recurrence relation
\begin{eqnarray*}
 \mathbb{E}[G_{n+1}^{i}(z)] = \frac{2n + 1 +z}{2n} \mathbb{E}[G_{n}^{i}(z)].
\end{eqnarray*}

\noindent Since $G_{1}^{1}(z) = G_{1}^{2}(z) =1$, the above recursive formula leads to
\begin{eqnarray} \label{pec1}
 \mathbb{E}[G_{n}^{1}(z)] = \mathbb{E}[G_{n}^{2}(z)] = \prod_{j=1}^{n-1} \frac{2j + 1 +z}{2j},
\end{eqnarray}

\noindent and for $3 \leq i \leq n+1$
\begin{eqnarray} \label{pec4}
 \mathbb{E}[G_{n}^{i}(z)] =  \left( \prod_{j=i-1}^{n-1} \frac{2j + 1 +z}{2j}  \right) \mathbb{E}[G_{i-1}^{i}(z)].
\end{eqnarray}

\noindent with the convention that $\prod_{j=n}^{n-1} \frac{2j + 1 +z}{2j} = 1$. We point out that $G_{n}^{i}(z) = 0$ for $n \leq i-2$.
We have the following technical result which will be crucial in the proof of Lemma \ref{lema3}.

\begin{lemma} \label{lema4}
 For $2 \leq i \leq n$, we have that
\begin{align*}
 & \mathbb{E}[G_{n}^{i}(z)Z^{i}(n,1)]   \\
 & ~~~ = \left( \prod_{j=i-1}^{n-1} \frac{2j+2+z}{2j} \right) \mathbb{E}[G_{i-1}^{i}(z)] + \sum_{k=i-1}^{n-1} \left( \prod_{j=k+1}^{n-1} \frac{2j+2+z}{2j} \right) \frac{1}{2k} \mathbb{E}[Z^{i}(k,1)],
\end{align*}
\noindent and
\begin{eqnarray*}
 \mathbb{E}[G_{n}^{1}(z)Z^{1}(n,1)] = \prod_{j=1}^{n-1} \frac{2j+2+z}{2j}  + \sum_{k=1}^{n-1} \left( \prod_{j=k}^{n-1} \frac{2j+2+z}{2j} \right) \frac{1}{2k} \mathbb{E}[Z^{1}(k,1)].
\end{eqnarray*}
\end{lemma}

\begin{proof}
We only prove the case when $2 \leq i \leq n$, the case $i=1$ follows exactly by the 
same argument. For $n \geq i-1 \geq 1$, we observe that $G_{n+1}^{i}(z) = G_{n}^{i}(z) + K^{i}_{n}(z)$ where 
 \begin{eqnarray*}
\mathbb{P}(K_{n}^{i}(z) = z^{k-1} | \mathcal{F}_{n}) = \left\{ \begin{array}{lcl}
              \frac{Z^{i}(n,k) + Z^{i}(n,k-1) }{2n} &  &k > 1\\
              \frac{Z^{i}(n,1)}{2n}  &  & k=1, \\
              \end{array}
    \right.
\end{eqnarray*}

\noindent and $Z^{i}(n+1,1) = Z^{i}(n,1) + B_{n}^{i}$ where
\begin{eqnarray*}
 \mathbb{P}(B_{n}^{i}= 1 | \mathcal{F}_{n}) = 1 -  \mathbb{P}(B_{n}^{i}= 0 | \mathcal{F}_{n}) = \frac{Z^{i}(n,1)}{2n}. 
\end{eqnarray*}

\noindent This yields
\begin{eqnarray*}
 \mathbb{E}(K_{n}^{i}(z)| \mathcal{F}_{n}) = \frac{1+z}{2n} G_{n}^{i}(z), \hspace*{3mm} \text{and} \hspace*{3mm} \mathbb{E}(B_{n}^{i}| \mathcal{F}_{n}) =  \mathbb{E}(K_{n}^{i}(z) B_{n}^{i} | \mathcal{F}_{n}) = \frac{Z^{i}(n,1)}{2n}.  
\end{eqnarray*}

\noindent Then, it follows that
\begin{eqnarray*}
 \mathbb{E}[G_{n+1}^{i}(z)Z^{i}(n+1,1)] & = & \mathbb{E}[ ( G_{n}^{i}(z) + K^{i}_{n}(z)) (Z^{i}(n,1) + B_{n}^{i}) |\mathcal{F}_{n}] \\ 
 & = & \frac{2n + 2 + z}{2n} \mathbb{E}[G_{n}^{i}(z)Z^{i}(n,1)] + \frac{1}{2n} \mathbb{E}[Z^{i}(n,1)].
\end{eqnarray*}

\noindent Since $Z^{i}(i-1,1) = 1$, this recursive formula yields to our result. 
\end{proof}

Next, we observe that for $1 \leq i \leq n+1$ the variable $Z^{i}(n,1)$ is the
degree of the vertex $i$ after the $n$-step, which first moment is given
by (see \cite{Mo})
\begin{eqnarray} \label{pec3}
 \mathbb{E}[Z^{1}(n,1)] = \prod_{j=1}^{n-1} \frac{2j + 1}{2j}, \hspace*{4mm} \text{and} \hspace*{4mm} \mathbb{E}[Z^{i}(n,1)] = \prod_{j=i-1}^{n-1} \frac{2j + 1}{2j}, \hspace*{3mm} \text{for} \hspace*{2mm} 2 \leq i \leq n+1
\end{eqnarray}

\noindent with the convention that $\prod_{j=n}^{n-1} \frac{2j + 1}{2j} = 1$. \\

We recall some technical results that will be useful later on. 
We have the following well-known inequality,
\begin{eqnarray} \label{pec2}
 1+x \leq e^{x}, \hspace*{6mm} x \in\mathbb{R}.
\end{eqnarray}

\noindent Then, we can easily deduce that
\begin{eqnarray} \label{pec6}
 \prod_{j=i-1}^{n-1} \frac{2j+2+z}{2j} \leq e^{\frac{2+z}{2}} \left(\frac{n-1}{i-1} \right)^{\frac{2+z}{2}} \hspace*{5mm} \text{and} \hspace*{5mm} \prod_{j=i-1}^{n-1} \frac{2j + 1}{2j} \leq e^{\frac{1}{2}} \left( \frac{n-1}{i-1} \right),
\end{eqnarray}

\noindent for $2 \leq i \leq n$. We recall also that by the Euler-Maclaurin formula we have that
\begin{eqnarray*}
 \sum_{j=1}^{n} \left(\frac{1}{j} \right)^{s} = \left(\frac{1}{n}\right)^{s-1}  + s \int_{1}^{n} \frac{\lfloor x \rfloor}{x^{s+1}} {\rm d} x, \hspace*{5mm} \text{with} \hspace*{5mm} s \in \mathbb{R} \setminus \{1 \},
\end{eqnarray*}

\noindent for $n \geq 1$. Then,
\begin{eqnarray} \label{pec7}
 \sum_{j=1}^{n} \left(\frac{1}{j} \right)^{s} \leq \left( 1 + \frac{s}{1-s} \right) n^{1-s}, \hspace*{5mm} \text{for} \hspace*{2mm} s \in (0,1),
\end{eqnarray}

\noindent and
\begin{eqnarray} \label{pec8}
 \sum_{j=1}^{n} \left(\frac{1}{j} \right)^{s} \leq \frac{s}{s-1}, \hspace*{5mm} \text{for} \hspace*{2mm} s > 1.
\end{eqnarray}



\begin{lemma} \label{lema5}
 There exists $z_{0} \in (0,1)$ such that
 \begin{eqnarray} \label{pec9}
  \mathbb{E}[ G_{i-1}^{i}(z_{0})] \leq (i-1)^{\frac{1+z_{0}}{2}}, \hspace*{5mm} \text{for} \hspace*{2mm} i \geq 2.
 \end{eqnarray} 
\end{lemma}

\begin{proof}
First, we focus on finding the correct $z_{0}$. For $i \geq 4$, let $v_{i}$ be the parent of the vertex $i$ which is distributed according to the law 
\begin{eqnarray*}
\mathbb{P} \left( v_{i} = j| T_{i-2}^{(0)} \right) = \frac{Z^{j}(i-2,1)}{2(i-2)}, \hspace*{5mm} j \in \{1,2, \dots, i-1\}.
\end{eqnarray*}

\noindent Then, we have that
\begin{eqnarray} \label{pec5}
 \mathbb{E}[G_{i-1}^{i}(z)] & = & 1+z \mathbb{E}[G^{v_{i}}_{i-2}(z)] \nonumber \\
                            & = & 1 + z \sum_{j=1}^{i-1} \mathbb{E}[ G_{i-2}^{j}(z) \mathbf{1}_{\{v_{i} = j\}}] \nonumber \\
                            & = & 1+\frac{z}{2(i-2)} \sum_{j=1}^{i-1} \mathbb{E}[ G_{i-2}^{j}(z) Z^{j}(i-2,1)].
\end{eqnarray}

\noindent We observe that Lemma \ref{lema4}, (\ref{pec3}) and (\ref{pec6}) imply after some computations that
\begin{align*}
 & \mathbb{E}[ G_{n-2}^{j}(z) Z^{j}(n-2,1)]  \\
 & ~ \leq e^{\frac{2+z}{2}} \left(\frac{n-3}{j-1} \right)^{\frac{2+z}{2}} \left( \mathbb{E}[ G_{j-1}^{j}(z)] + \frac{e^{\frac{1}{2}}}{2} (j-1)^{-\frac{3+z}{2}} \sum_{k=j-1}^{n-4} \left( \frac{1}{k} \right)^{\frac{3+z}{2}} \right)+ \frac{1}{2}e^{\frac{1}{2}}  \left(\frac{1}{(n-3)(j-1)} \right)^{\frac{1}{2}}  
\end{align*}

\noindent for $2 \leq j \leq n-3$. Then the inequalities (\ref{pec7})  and (\ref{pec8}) imply that
\begin{equation}
\begin{aligned}  
 & \sum_{j=2}^{n-3} \mathbb{E}[ G_{n-2}^{j}(z) Z^{j}(n-2,1)] \\
 & ~ \leq e^{\frac{2+z}{2}} (n-3)^{\frac{2+z}{2}} \left( \sum_{j=2}^{n-3} \left( \frac{1}{j-1} \right)^{\frac{2+z}{2}} \mathbb{E}[ G_{j-1}^{j}(z)] + e^{\frac{1}{2}} \frac{3 + z}{1+z} (n-3)^{\frac{1}{2}} \right)+ \frac{e^{\frac{1}{2}}}{(n-3)^{\frac{1}{2}}},
\end{aligned}
\label{pec10}
\end{equation}

\noindent for $n \geq 5$. Similarly, one gets that
\begin{eqnarray} \label{pec11}
 \mathbb{E}[ G_{n-2}^{1}(z) Z^{1}(n-2,1)]  \leq e^{\frac{2+z}{2}} (n-3)^{\frac{2+z}{2}} + \frac{e^{\frac{3+z}{2}}}{2}\frac{3+z}{1+z} (n-3)^{\frac{2+z}{2}}
\end{eqnarray}

\noindent and
\begin{eqnarray} \label{pec12}
 \mathbb{E}[ G_{n-2}^{n-2}(z) Z^{n-2}(n-2,1)]  \leq e^{\frac{2+z}{2}} \mathbb{E}[G_{n-3}^{n-2}(z)] + \frac{1}{2} (n-3)^{-1},
\end{eqnarray}

\noindent for $n \geq 4$. Next, we define the functions
\begin{eqnarray*}
 A_{n}^{1}(z) = \left( e^{\frac{2+z}{2}} + \frac{e^{\frac{1+z}{2}}}{2}\frac{3+z}{1+z} \right) (n-3)^{-\frac{1}{2}}, \hspace*{5mm}  A_{n}^{2}(z) = \left( e^{\frac{2+z}{2}} + \frac{1}{2} (n-3)^{-\frac{3+z}{2}} \right) (n-3)^{-1}
\end{eqnarray*}

\noindent and
\begin{eqnarray*}
 A_{n}^{3}(z) = 2e^{\frac{2+z}{2}} + e^{\frac{1}{2}}(n-3)^{-\frac{4+z}{2}} + e^{\frac{3+z}{2}} \frac{3 + z}{1+z} ,
\end{eqnarray*}

\noindent for $n \geq4$ and $z \in (0,1)$. Then one can find $z_{0} \in (0,1)$ such that
\begin{eqnarray*}
 3^{-\frac{1+z_{0}}{2}} + \frac{z_{0}}{2} \left(A_{4}^{1}(z_{0}) + A_{4}^{2}(z_{0}) + A_{4}^{3}(z_{0}) +  \frac{1}{2} \right) \leq 1.
\end{eqnarray*}

\noindent Now, we proceed to prove by induction (\ref{pec9}) with $z_{0} \in (0,1)$ such that the previous inequality is satisfied. For $i =2,3$, it must be clear since
\begin{eqnarray*}
 \mathbb{E}[G_{1}^{2}(z_{0})] = 1 \hspace*{4mm} \text{and} \hspace*{4mm} \mathbb{E}[G_{2}^{3}(z_{0})] = 1 + z_{0}.
\end{eqnarray*}

\noindent Suppose that it is true for $i=n -1\geq 2$. 
We observe from (\ref{pec5}) and the inequalities (\ref{pec10}), (\ref{pec11}) and (\ref{pec12}) that
\begin{eqnarray*}
 \mathbb{E}[G_{n-1}^{n}(z_{0})] & \leq & 1 +  (n-1)^{\frac{1+z_{0}}{2}} \frac{z_{0}}{2} \left( A_{n}^{1}(z_{0}) + A_{n}^{2}(z_{0}) + \frac{1}{2} + A_{n}^{3}(z_{0}) \right) \\
 & \leq &    (n-1)^{\frac{1+z_{0}}{2}} \left( 3^{-\frac{1+z_{0}}{2}} + \frac{z_{0}}{2} \left( \frac{1}{2} + A_{4}^{1}(z_{0}) + A_{4}^{2}(z_{0}) + A_{4}^{3}(z_{0}) \right)\right) \\
 & \leq &   (n-1)^{\frac{1+z_{0}}{2}},
\end{eqnarray*}

\noindent the second inequality is because the functions $A_{n}^{1}(\cdot)$, $A_{n}^{2}(\cdot)$ and $A_{n}^{1}(\cdot)$ are decreasing with respect to $n$ and the last one is by our choice of $z_{0}$.
\end{proof}

Finally, we have all the ingredients to prove Lemma \ref{lema3}.

\begin{proof}[Proof of Lemma \ref{lema3}] We deduce from the inequality 
(\ref{pec2}) that for $n \geq 2$ we have 
\begin{eqnarray*} 
 \prod_{j=i-1}^{n-1} \frac{2j+1+z}{2j} \leq e^{\frac{1+z}{2}} \left(\frac{n-1}{i-1} \right)^{\frac{1+z}{2}} \hspace*{5mm} \text{for} \hspace*{2mm} i \geq 2.
\end{eqnarray*}

\noindent We consider $z_{0} \in (0,1)$ such that equation (\ref{pec9}) in 
Lemma \ref{lema5} is satisfied. Then from (\ref{pec1}) and (\ref{pec4}) we have that
\begin{eqnarray*}
 \mathbb{E}[G_{n}^{1}(z_{0})] \leq e^{\frac{1+z_{0}}{2}} n^{\frac{1+z_{0}}{2}} \hspace*{5mm} \text{for} \hspace*{2mm} i \geq 1 \hspace*{2mm} \text{and} \hspace*{2mm} n \geq 1,
\end{eqnarray*}

\noindent which is our claim.
\end{proof}

\paragraph{Acknowledgements.} I am grateful to Jean Bertoin for introducing me to the topic and 
for many fruitful discussions. I would also like to thank the two anonymous 
referees whose suggestions and remarks helped to improve this paper.\\

This work is supported by the Swiss National Science Foundation 200021\_144325/1

 
\providecommand{\bysame}{\leavevmode\hbox to3em{\hrulefill}\thinspace}
\providecommand{\MR}{\relax\ifhmode\unskip\space\fi MR }
\providecommand{\MRhref}[2]{%
  \href{http://www.ams.org/mathscinet-getitem?mr=#1}{#2}
}
\providecommand{\href}[2]{#2}

\end{document}